\documentclass[11pt,centertags,oneside]{article}
\usepackage{amsmath,amstext,amsthm,amscd,typearea}
\usepackage{amssymb}
\usepackage{a4wide}
\usepackage[mathscr]{eucal}
\usepackage{mathrsfs}
\usepackage{typearea}
\usepackage{charter}
\usepackage{pdfsync}
\usepackage{url}

\usepackage{xcolor}

\usepackage[a4paper,width=15.2cm,top=4cm,bottom=4cm]{geometry}

\numberwithin{equation}{section}

\newtheorem{theorem}{Theorem}[section]
\newtheorem{definition}[theorem]{Definition}
\newtheorem{proposition}[theorem]{Proposition}
\newtheorem{corollary}[theorem]{Corollary}
\newtheorem{lemma}[theorem]{Lemma}
\newtheorem{remark}[theorem]{Remark}

\newtheorem{example}[theorem]{Example}

\newtheorem{problem}[theorem]{Problem}

\newcommand{\cali}[1]{\mathscr{#1}}

\newcommand{\vol}{\mathop{\mathrm{vol}}}

\newcommand{\ddc}{dd^c}

\newcommand{\Sing}{\text{\normalfont Sing}}

\newcommand{\B}{\mathbb{B}}
\newcommand{\C}{\mathbb{C}}

\newcommand{\N}{\mathbb{N}}

\renewcommand\P{\mathbb{P}}

\title{\bf  Volume of components of Lelong upper-level sets}
\providecommand{\keywords}[1]{\textbf{\textit{Keywords:}} #1}
\providecommand{\subject}[1]{\textbf{\textit{Mathematics Subject Classification 2010:}} #1}

\author{Do Duc Thai and Duc-Viet Vu}

\newcommand{\Addresses}{{
		\bigskip
		\footnotesize
		\textsc{Duc-Viet Vu, University of Cologne, Division of Mathematics, Department of Mathematics and Computer Science, Weyertal 86-90, 50931, K\"oln,  Germany.}
		\noindent
		\par\nopagebreak
		\noindent
		\textit{E-mail address}: \texttt{vuviet@math.uni-koeln.de}
\newline
		
		\textsc{Do Duc Thai, Department of Mathematics, Hanoi National University of Education, 136 XuanThuy str., Hanoi, Vietnam.}
		\noindent
		\par\nopagebreak
		\noindent
		\textit{E-mail address}: \texttt{doducthai@hnue.edu.vn}	
}}

\date{\today}
\begin{document}
\maketitle
\begin{abstract} We prove an upper bound for the volume of maximal analytic sets on which the generic Lelong number of a closed positive current is positive. As a particular case, we give a uniform upper bound on the volume of the singular locus of an analytic set in terms of its volume  on a compact K\"ahler manifold.  
\end{abstract}
\noindent
\keywords Density current, Lelong number, Lelong filtration, Dinh-Sibony product.
\\

\noindent
\subject{32U15}, {32Q15}.



\section{Introduction}

Let $T$ be a closed positive current on a compact K\"ahler manifold $X$ of dimension $n$. For every irreducible analytic subset $V$ in $X$, we denote by $\nu(T,V)$ the generic Lelong number of $T$ along $V$. For basics of Lelong numbers, we refer to \cite{Demailly_ag}. Let $W$ be an irreducible analytic subset of dimension $m$ on $X$. Let $\cali{V}_{T,W}$ be the set of (proper) irreducible analytic subsets $V$ of $W$ such that $\nu(T,V)>0$, and $V$ is maximal with respect to the inclusion of sets, \emph{i.e,} if $V'$ is another proper irreducible analytic set in $W$ so that $\nu(T,V')>0$ and $V \subset V'$, then $V'=V$. Such a $V$ is called \emph{a  component of Lelong upper-level set} of $T$ on $W$.

Let $\omega$ be a fixed K\"ahler form on $X$.  We denote by $\|T\|$ the mass norm on $T$ given by $\|T\|:= \int_X T \wedge \omega^{n-p},$  
where $T$ is of bi-degree $(p,p)$. Similarly we define $\vol(V):= \int_V \omega^{\dim V}$ for every irreducible analytic subset $V$ in $X$. Usually the volume of an analytic set involves a positive multiplicative dimensional constant. Since the dimensions of analytic sets in consideration are bounded by $n= \dim X$, we opt to define $\vol(V)$ as above so that the volume of $V$ is simply equal to the mass norm of the current of integration along $V$. This is only for a practical purpose and is not essential.     

\begin{theorem} \label{the-self-inters}  There exists a constant $c>0$ independent of $T$ and $W$ such that 
\begin{align}\label{ine-dimVlself-interdesire}
 \sum_{V \in \cali{V}_{T,W}} \big(\nu(T,V)- \nu(T,W)\big)^{m-\dim V} \vol(V) \le c \vol(W)\|T\| (1+\|T\|)^{m-1}.
 \end{align}
\end{theorem}


The constant $c$ in Theorem \ref{the-self-inters} depends only on $X$ and $\omega$ but it is non-explicit: it comes from upper bounds for density currents (see Theorem \ref{th-dieukienHVconictangenetcurrent} below).

In Theorem \ref{the-self-inters}, it is necessary to consider the ``relative"  generic Lelong number $\nu(T,V)- \nu(T,W)$ of $T$ along $V$ in $W$ in place of $\nu(T,V)$, see Example \ref{ex-phaidungrelative} in the end of the paper. 

The irreducibility of $W$ is not absolutely necessary. Nevertheless if $W$ is not irreducible, the inequality (\ref{ine-dimVlself-interdesire}) must be modified a bit to take into account the generic Lelong number of $T$ along irreducible components of $W$. 
Theorem \ref{the-self-inters}  gives in particular an upper bound for the volume of a maximal irreducible analytic subset of $W$ on which $T$ has a strictly positive ``relative" generic Lelong number. 

Let $T$ be of bi-degree $(p,p)$. If $p=0$, then $T$ is a constant function, and in this case, Theorem \ref{the-self-inters} is clear because $\cali{V}_{T,W}$ is empty. If $p=n$, then $T$ is a measure on $X$ and the only maximal Lelong level sets for $T$ are points where $T$ has positive mass. In this case, the desired inequality (\ref{ine-dimVlself-interdesire}) follows from the fact that $\vol(W)$ is bounded below by a constant $c_0>0$ independent of $W$ (see Lemma \ref{le-lowerboundvolumeana} below) and the elementary inequality 
$$\sum_{j=1}^\infty a_j^m \le \big(\sum_{j=1}^\infty a_j\big)\big(1+ \sum_{j=1}^\infty a_j\big)^{m-1}$$
for positive numbers $a_j$. Thus the non-trivial case is when $1\le p \le n-1$. Since a current of bi-degree $(p,p)$ with $p \ge 1$ on $X$ has zero generic Lelong number along $X$, by taking $W:=X$ in Theorem \ref{the-self-inters}, one obtains 

\begin{corollary} \label{cor-apdungRbangT} Let $T$ be a closed positive current of bi-degree $(p,p)$ with $p \ge 1$ on $X$. Let $\cali{V}_T$ be the set of maximal irreducible analytic subsets $V$ in $X$ along which the generic Lelong number of $T$ is strictly positive. Then we have 
$$\sum_{V \in \cali{V}_{T}} \big(\nu(T,V)\big)^{n-\dim V} \vol(V) \le c \|T\| (1+\|T\|)^{n-1},$$
for some constant $c>0$ independent of $T$. 
\end{corollary}

By \cite{Vigny-LelongSkoda}, for every closed positive current  $T$ of bi-degree $(p,p)$ with $1 \le p \le n-1$, there exists a closed positive current of bi-degree $(1,1)$ whose Lelong numbers are equal to those of $T$ everywhere. Hence it makes no difference in the above results if we assume that $T$ is of bi-degree $(1,1)$.  Thus Corollary \ref{cor-apdungRbangT} follows indeed from \cite[Theorem 4.1]{Vu_lelong-bignef-quantitative}, see Remark \ref{re-lelongbignef} below. 

Corollary \ref{cor-apdungRbangT} is related to  Demailly's self-intersection inequality (\cite[Theorem 1.7]{Demailly_regula_11current}) and other higher bi-degree versions of it in \cite{Meo-auto-inter,Parra,Vigny-LelongSkoda} . If one applies directly \cite[Theorem 1.7]{Demailly_regula_11current} or results in \cite{Meo-auto-inter,Parra,Vigny-LelongSkoda}, then one only gets an upper bound for the volume of maximal Lelong sets of maximal dimension. Precisely, let $k:= \max_{V \in \cali{V}_T} \dim V$ and let $\cali{V}'_{T}$ be the subset of $\cali{V}_T$ containing $V$ of dimension $k$. By these last references, there holds
$$\sum_{V \in \cali{V}'_{T}} \big(\nu(T,V)\big)^{n-k} \vol(V) \le c\|T\|^{n-k},$$
for some constant $c>0$ independent of $T$. 
However \cite[Theorem 1.7]{Demailly_regula_11current} is not sufficient to treat maximal Lelong level sets which are not of maximal dimension $k$ because 
these sets are not read in the process of considering jumping values of  Lelong level sets if their Lelong numbers are less than or equal to $\max_{V\in \cali{V}'_T} \nu(T,V)$, see Remark \ref{re-demailly-self} at the end of the paper for an example illustrating this issue.  Hence the novelty in Corollary \ref{cor-apdungRbangT} (and Theorem \ref{the-self-inters}) is that it bounds the volume of every maximal Lelong level set, see also Corollary \ref{cor-volumesingularset} below.


Let $W$ be analytic set in $X$. We don't require $W$ to be irreducible.  We define a sequence of analytic subsets as follows. Let $W_0:=W$ and $W_j$ to be the singular locus of $W_{j-1}$ for $j \ge 1$. The process ends at a finite time when $W_j$ is smooth because of a dimension reason. Hence we obtain a finite sequence $(W_j)_{1 \le j \le k_W}$. Note $k_W \le \dim W$.  We call $(W_j)_{1 \le j \le k_W}$ the \emph{singularity filtration} of $W$.

Let $T:= [W]$ be the current of integration along $W$. Recall that $\nu(T,x)$ is the multiplicity of $W$ at $x$ for every $x \in X$, and $\nu(T,x)= 1$ if $x \in W_0 \backslash W_1$ and $\nu(T,x) \ge 2$ if $x \in W_1$ (see \cite[Page 120]{chirka}). The analytic set $W_j$ might not be irreducible and hence can contain several irreducible components of different dimensions. Write $W_j= \cup_{k=1}^l W_{jk}$ which is the decomposition of $W_j$ into irreducible components. We define 
$$\vol(W_j):= \sum_{k=1}^l \vol(W_{jk})= \sum_{k=1}^l \int_{W_{jk}} \omega^{\dim W_{jk}}.$$
As an application of Theorem \ref{the-self-inters}, we obtain


\begin{corollary}\label{cor-volumesingularset} There exists a constant $c>0$ independent of $W$ such that 
\begin{align}\label{ine-vol-sing}
\vol(W_1) \le c \vol(W)(1+ \vol(W))^{m-1}
\end{align}
and 
\begin{align} \label{ine-vol-sing2}
\sum_{j=1}^{k_W} \vol(W_j) \le c \vol(W)(1+ \vol(W))^{(m+2)^m},
\end{align}
where $m:= \dim W$.
\end{corollary}


If one uses \cite[Theorem 1.7]{Demailly_regula_11current}, then one only obtains a similar upper bound for the sum of volumes of irreducible components of $W_1$ of maximal dimension. The reason as explained above is that some maximal Lelong level sets are not taken into account in the sequence of jumping numbers for the Lelong level sets $E_c$ in \cite[Theorem 1.7]{Demailly_regula_11current}.

We can refine Corollary \ref{cor-apdungRbangT} a bit more as follows.   \emph{A Lelong filtration associated to $T$} is a finite sequence of (non-empty) irreducible analytic subsets $\mathbf{V}:=(V_j)_{0 \le j \le l}$ in $X$ such that $V_0 \subsetneq V_1 \subsetneq \cdots \subsetneq V_l$, and for $1 \le j \le l+1$, $V_{j-1}$ is a maximal irreducible analytic subset of $T$ on  $V_j$ such that $\nu(T,V_{j-1})> \nu(T,V_j)$, where $V_{l+1}:= X$. The number $l$ is called \emph{the length of the filtration $\mathbf{V}$}.  For each Lelong filtration $\mathbf{V}$, we put
$$\vol(\mathbf{V}):= \vol(V_0) \prod_{s=0}^l \big(\nu(T,V_s)- \nu(T,V_{s+1})\big)^{\dim V_{s+1}- \dim V_s} .$$ 
A Lelong filtration $\mathbf{V}= (V_j)_{1 \le j \le m}$ is \emph{maximal} if $\nu(T,V_0)= \nu(T,x)$ for every $x \in V_0$, in other words, one can not add one more analytic subset to $\mathbf{V}$ to make it longer.  Let $\cali{F}_T$ be the set of Lelong filtrations of $T$. 


\begin{theorem} \label{theo-apdungRbangT} Let $T$ be a closed positive current of bi-degree $(p,p)$ with $p \ge 1$ on $X$. Then we have 
$$\sum_{\mathbf{V} \in \cali{F}_{T}} \vol(\mathbf{V}) \le c \|T\| (1+\|T\|)^{n^2},$$
for some constant $c>0$ independent of $T$. 
\end{theorem}

Since the self-intersection of Demailly mentioned above has found applications in the study of equidistribution of pre-images of subvarieties in complex dynamics (see \cite{DinhSibony-convergence-Green,Taflin}), we expect that Theorem \ref{the-self-inters} could be also useful  for this question. Corollary \ref{cor-volumesingularset} is by the way of independent interest in its own right.

The difficulty in proving Theorem \ref{the-self-inters} is that the intersection of copies of $T$ with the current of integration along $W$ is not well-defined in the classical sense. To solve this problem, we need to use both the analytic regularisation of psh functions by Demailly (\cite{Demailly_regula_11current}) and the theory of density currents by Dinh-Sibony (\cite{Dinh_Sibony_density}). The fact that the notion of density currents generalizes the classical intersection of currents of bi-degree $(1,1)$ (see \cite{VietTuanLucas}) is also crucial in our proof.

We suspect that the inequality in Corollary \ref{cor-apdungRbangT} is more or less the best possible relation between  $\nu(T,V)$ for maximal $V$ if we don't take into account the cohomology class of $T$. More precisely we would like to pose the following question. 

\begin{problem} \label{problem-nguoc}  Let $(V_j)_{j \in \N}$ be a countable family of irreducible analytic subsets in $X$ together with a sequence of strictly positive real numbers $(\lambda_j)_{j \in \N}$ such that  
$V_j \not \subset V_{j'}$ for every $j \not = j'$, and 
$$\sum_{j=1}^\infty \lambda_j^{n- \dim V_j} \vol(V_j)< \infty.$$
Does there exist a closed positive current $T$ of bi-degree $(1,1)$ such that 
$\nu(T,V_j) \ge \lambda_j$ and $V_j$ is a maximal analytic subset along which $T$ has strictly positive generic Lelong number for every $j$?
Furthermore, if  $\alpha$ is a pseudoeffective $(1,1)$-class, under which additional assumptions on $(V_j,\lambda_j)$ there exists  a closed positive current $T$ in $\alpha$ so that $T$ satisfies the above property? 
\end{problem}

The problem was solved in \cite[Corollary 6.8]{Demailly-numerical-criterion} in the case where $\alpha$ is big and nef, and $(V_j)_{j \in \N}$ is a countable family of points in $X$. The necessary and sufficient condition in this case is 
$$\sum_{j=1}^\infty \lambda_j^n \le \int_X\alpha^n.$$

Problem \ref{problem-nguoc} can be put in a  broader context as follows. Consider a big cohomology $(1,1)$-class $\alpha$ and an analytic subset $V$ in $X$. One wants to study the space of currents in $\alpha$ having controlled singularities along $V$ (in some reasonable ways). Such a question is relevant in the theory of complex Monge-Amp\`ere equations big cohomology classes where this equation was solved with given prescribed singularities (see \cite{Darvas-Lu-DiNezza-singularity-metric,Lu-Darvas-DiNezza-logconcave,Lu-Darvas-DiNezza-mono,Vu_Do-MA}).    We refer also to \cite{Coman-Marinescu-Nguyen,Darvas-Xia} for some results around this topic in the case where the class is integral.   
\\

\noindent
\textbf{Acknowledgement.} We thank Tien-Cuong Dinh and Gabriel Vigny for fruitful discussions. The research of D.-V. Vu is partly funded by the Deutsche Forschungsgemeinschaft (DFG, German Research Foundation)-Projektnummer 500055552. The research of Do Duc Thai is supported by the ministry-level project B2022 - SPH - 03.

\section{Proofs of Theorem \ref{the-self-inters} and its consequences} \label{sec-maximalLelong}

We first recall some basic properties of density currents. The last notion was introduced in \cite{Dinh_Sibony_density}. 

Let $X$ be a complex K\"ahler manifold of dimension $n$ and $V$ a smooth complex submanifold of $X$ of dimension $l.$  Let $T$ be a closed positive $(p,p)$-current on $X,$ where $0 \le p \le n.$  Denote by $\pi: E\to V$ the normal bundle of $V$ in $X$ and $\overline E:= \P(E \oplus \C)$ the projective compactification of $E.$ By abuse of notation, we also use $\pi$ to denote the natural projection from $\overline E$ to $V$. 

Let $U$ be an open subset of $X$ with $U \cap V \not = \varnothing.$  Let $\tau$ be  a smooth diffeomorphism  from $U$ to an open neighborhood of $V\cap U$ in $E$ which is identity on $V\cap U$ such that  the restriction of its differential $d\tau$ to $E|_{V \cap U}$ is identity.  Such a map is called \emph{an almost-admissible map}.  When $U$ is a small enough tubular neighborhood of $V,$ there always exists an almost-admissible map $\tau$ by \cite[Lemma 4.2]{Dinh_Sibony_density}. In general, $\tau$ is not holomorphic.  When $U$ is a small enough local chart, we can choose a holomorphic almost-admissible map by using suitable holomorphic coordinates on $U$.   For $\lambda \in \C^*,$ let $A_\lambda: E \to E$ be the multiplication by $\lambda$ on fibers of $E.$ 

\begin{theorem} \label{th-dieukienHVconictangenetcurrent} (\cite[Theorem 4.6]{Dinh_Sibony_density})
Let $\tau$ be an almost-admissible map defined on a tubular neighborhood of $V$. Then,  the family $(A_\lambda)_* \tau_* T$ is of mass bounded uniformly in $\lambda$ on compact sets, and  if $S$ is a limit current of the last family as $\lambda \to \infty$, then  $S$ is a current on $E$ which can be extended trivially through $\overline E \backslash E$ to  a closed positive current on $\overline E$  such that the cohomology class $\{S\}$ of $S$ in $\overline E$ is independent of the choice of $S$, and $\{S\}|_V= \{T\}|_V$,  and $\|S\| \le C \|T\|$ for some constant $C$ independent of $S$ and $T$.
\end{theorem}

We call $S$  \emph{a tangent current to $T$ along $V$}.  By \cite[Theorem 4.6]{Dinh_Sibony_density} again,  if 
$$S=\lim_{k\to \infty} (A_{\lambda_k})_* \tau_* T$$ for some sequence $(\lambda_k)_k$ converging to $\infty$, then  for every open subset $U$ of $X$ and  every almost-admissible map $\tau': U' \to E$ , we also have  
$$S=\lim_{k\to \infty} (A_{\lambda_k})_* \tau'_* T.$$
This is equivalent to saying that tangent currents are independent of the choice of almost-admissible maps.  By this reason, the sequence $(\lambda_k)_k$ is called \emph{a defining sequence} of $T_\infty.$  
In practice, to study tangent currents, we usually choose $\tau'$ to be a holomorphic change of coordinates.

Let $m\ge 2$ be an integer. Let $T_j$ be a closed positive current  of bi-degree $(p_j, p_j)$ for $1 \le j \le m$ on $X$ and let  $T_1 \otimes \cdots \otimes T_m$ be the tensor current of $T_1, \ldots, T_m$ which is a current on $X^m.$  A \emph{density current} associated to $T_1, \ldots,  T_m$ is a tangent current to $\otimes_{j=1}^m T_j$ along the diagonal $\Delta_m$ of $X^m.$ Let $\pi_m: E_m \to \Delta$ be the normal bundle of $\Delta_m$ in $X^m$. When $m=2$ and $T_2 =[V]$, the density currents of $T_1$ and  $T_2$ are naturally identified with the  tangent currents to $T_1$ along $V$ (see \cite{Dinh_Sibony_density}  or \cite[Lemma 2.3]{Vu_density-nonkahler}).   


We say that the \emph{Dinh-Sibony product} $T_1 \curlywedge \cdots \curlywedge T_m$ of $T_1, \ldots, T_m$ is well-defined  if $\sum_{j=1}^m p_j \le n$ and  there is only one density current associated to $T_1, \ldots, T_m$ and this current is  the pull-back by $\pi_m$ of a current $S$ on $\Delta_m$. 
We define $T_1 \curlywedge \cdots \curlywedge T_m$ to be $S$. 


   
We note that one can still define density currents on general complex manifolds; see \cite{VietTuanLucas,Viet_Lucas,Viet-density-nonpluripolar} for details.  We recall the following from \cite{VietTuanLucas}.

Let $R$ be a positive closed current and $v$ be a psh function on an open subset $\Omega$ in $\C^n$. If $v$ is locally integrable with respect to (the trace measure) of $R$, we define
\begin{equation} \label{eq:def-bedford-taylor}
\ddc v \wedge R := \ddc (v R),
\end{equation}
see \cite{Bedford_Taylor_82}.

For a collection $v_1, \ldots, v_{s}$ of  psh functions, we can apply the above definition recursively, as long as the integrability conditions are satisfied. 

\begin{definition} \label{def-bedford-taylor} We say that the intersection  of $\ddc v_1, \ldots, \ddc v_s,R$ is \emph{classically well-defined} if for every non-empty subset $J=\{j_1, \ldots, j_k\}$ of $\{1, \ldots, s\}$, we have that $v_{j_k}$ is locally integrable with respect to the trace measure of $R$ and inductively,  $v_{j_r}$ is locally integrable with respect to the trace measure of $\ddc v_{j_{r+1}} \wedge \cdots \wedge \ddc v_{j_k} \wedge R$ for $r=k-1, \ldots, 1$, and the product $\ddc v_{j_{1}} \wedge \cdots \wedge \ddc v_{j_k} \wedge R$ is continuous under decreasing sequences of psh functions.
\end{definition}

The above definition is generalized in an obvious way to the case where $\Omega$ is replaced by a complex manifold. If $X$ is a complex manifold and $T_1,\ldots,T_m$ are closed positive currents of bi-degree $(1,1)$ on $X$ and $R$ is closed positive current on $X$, then we say that the intersection  $T_1\wedge \cdots T_m \wedge R$ of $T_1,\ldots, T_m,R$ is \emph{classically well-defined} if they are so locally on $X$.

The reason for Definition \ref{def-bedford-taylor} is that in all of well-known standard situations  warranting the well-definedness of the intersection of currents (see, e.g, \cite{Demailly_ag}), the intersection satisfies the continuity under decreasing sequences. In particular if $T_j$ has locally bounded potentials outside some analytic subset $V_j$ in $X$ and $R:= [W]$ the current of integration along an analytic set $W$, and 
$$\dim (V_{j_1} \cap \cdots \cap V_{j_s} \cap W) \le \dim W- s,$$
 for every $1 \le j_1 < \cdots < j_s \le m$, then $T_1 \wedge \cdots \wedge T_m \wedge R$ is classically well-defined.

\begin{theorem}\label{th-density11agreewithclassical} (\cite[Corollary 3.11]{VietTuanLucas}) Let $X$ be a complex manifold. Let $T_1, \ldots, T_m$ be closed positive currents of bi-degree $(1,1)$ on $X$ and let $R$ be a closed positive current on $X$ such that $T_1 \wedge \cdots \wedge T_m \wedge R$ is classically well-defined. Then the Dinh-Sibony product of $T_1, \ldots, T_m,R$ is well-defined and equal to $T_1 \wedge \cdots \wedge T_m \wedge R$. 
\end{theorem}

Note that the last result was proved in \cite{Dinh_Sibony_density} when local potentials of $T_j$'s are continuous. We recall the following.

\begin{lemma} \label{le-countablelelong} The set of Lelong numbers of an arbitrary closed positive current on an open subset in $\C^n$ is countable. 
\end{lemma}

\proof For readers' convenience, we present below a proof communicated to us by Tien-Cuong Dinh. It is more or less a direct consequence of Siu's analyticity of Lelong upper level sets (\cite{Siu}). Let $T$ be a closed positive current on $X$.  We show that for every analytic set $V$ in $X$, the set $\{x \in V: \nu(T,x)>0\}$ is at most countable. The desired assertion is the case where $V=X$. We prove this claim by induction on the dimension of $V$. If $\dim V=0$, there is nothing to prove. Assume that the claim holds for every analytic set of dimension $<m$. Let $V$ be now of dimension $m$. For $k>0$ integer, consider 
$$E_{1/k}:= \{x \in V: \nu(T,x) \ge \nu(T,V)+ 1/k\}$$
 which is a proper analytic subset of $V$ by Siu's analyticity theorem \cite{Siu}. Applying the induction hypothesis to $E_{1/k}$, one obtains that the set of Lelong numbers of $T$ on $E_{1/k}$ is at most countable. Since there is a countable number of $E_{1/k}$, we infer the claim for $V$. This finishes the proof. 
\endproof

The following result is also standard.

\begin{lemma}\label{le-lowerboundvolumeana} Let $X$ be a compact complex manifold of dimension $n$ and $\omega$ be a Hermitian form on $X$. Then  there exists a constant $C>0$ such that for every $0 \le p \le n$ and  every closed positive $(p,p)$-current $T$ and every $x \in X$, we have 
\begin{align}\label{ine-chanduoivoluelelong}
\int_X T \wedge \omega^{n-p} \ge C \nu(T, x).
\end{align}
In particular there exists a constant $C>0$ such that if $W$ is an irreducible analytic subset of dimension $m$ on $X$, then the $\vol(W):= \int_W \omega^m \ge C$.
\end{lemma}

\proof Cover $X$ be a finitely many local charts $U_1,\ldots, U_m$. We can assume furthermore that $\overline U_j$ is contained in a bigger local chart $U'_j$ for $1 \le j \le m$. Let $x \in X$. Then $x \in U_j$ for some $j$. We work locally  on $U_j \subset \C^n$. Let $\omega_0$ be the standard K\"ahler form on $\C^n$. Note that  there is a constant $\delta>0$ independent of $x$ such that $\B(x,2\delta) \subset U'_j$.  By the definition of Lelong numbers, we have 
$$\int_{\B(x,\delta)} T \wedge \omega_0^{n-p} \ge C_n \delta^{2(n-p)} \nu(T,x),$$
where $C_n$ is an explicit dimensional constant. Since $\omega$ is Hermitian, we obtain a constant $C_1>0$ such that $C_1\omega \ge \omega_0$. Thus
$$\int_{\B(x,\delta)} T \wedge \omega_0^{n-p} \le C_1^{n-p} \int_{\B(x,\delta)} T \wedge \omega^{n-p} \le C_1^{n-p} \int_X T \wedge \omega^{n-p}.$$
Hence (\ref{ine-chanduoivoluelelong}) follows. As for the second desired assertion, we recall that $\nu([W],x)$ is equal to the multiplicity of $x$ in $W$ which is a positive integer. Hence $\nu([W],x) \ge 1$ for every $x\in W$. The uniform lower bound for the volume of $W$ thus follows.  
\endproof

\begin{proof}[Proof of Theorem \ref{the-self-inters}]   By \cite[Theorem 3.1]{Vigny-LelongSkoda}, we can assume that $T$ is of bi-degree $(1,1)$. \\

\noindent
\textbf{Step 1.} We assume first that $T$ has analytic singularities in $X$ (i.e., potentials of $T$ are locally the sum of a bounded function and the logarithmic of a sum of holomorphic functions) and $\nu(T,W)=0$. We will explain how to remove this assumption at the end of the proof.

For an irreducible analytic set $V$, we denote by $[V]$ the current of integration along $V$. Denote by $L(T)$  the unbounded locus of $T$ (for a definition, see \cite[Page 150]{Demailly_ag}). In our present setting, the last set is an analytic subset of $X$.  Since $\nu(T,W)=0$, we have $W\not \subset L(T)$.  Let $0 \le l \le  m-1$ be an integer. Let $\cali{V}_{T,W}^l$ be the set of $V \in \cali{V}_{T,W}$ such that $\dim V =l$. The last set is finite because $T$ has analytic singularities.  Note that 
$$\cali{V}_{T,W}= \cup_{l=0}^{m-1} \cali{V}_{T,W}^{l}.$$  
Let $V \in \cali{V}_{T,W}^l$.  By the maximality of $V$, we see that $V$ is an irreducible component of $L(T) \cap W$.  Thus, since $\dim V =l$, we deduce that the intersection $T^{m-l} \wedge [W]$ is classically well-defined on some open neighborhood  $U_V$ of $V \backslash \Sing\big(L(T)\cap W\big)$, where $\Sing\big(L(T) \cap W\big)$ denotes the singular locus of the analytic set $L(T)\cap W$. 
 
Let $Q:= T^{m-l} \wedge [W]$ on $U_V$. Now, using a comparison result on Lelong numbers (\cite[Page  169]{Demailly_ag}) gives that 
$$Q \ge   \big(\nu(T,V)\big)^{m-l}\,[V]$$
 on $U$. Letting $V$ run over every element of $\cali{V}_{T,W}^l$, we infer that $Q$ is well-defined on some open neighborhood $U$ of 
$$\bigcup_{V \in \cali{V}_{T,W}^l} V  \backslash \Sing\big(L(T)\cap W\big),$$
 and
\begin{align}\label{ine-dimVlself-interQ}
Q \ge \sum_{V \in \cali{V}_{T,W}^l}  \big(\nu(T,V)\big)^{m-l}\, [V]
\end{align}
on $U$.   We now estimate the mass of $Q$.  Let $\Delta_{m-l+1}$ be the diagonal of $X^{m-l+1}$ and $\pi: E \to \Delta_{m-l+1}$ the normal bundle of $\Delta_{m-l+1}$ in $X^{m-l+1}$.  Let $\tilde{Q}$ be a density current associated to  $T, \ldots, T$ ($m-l$ times), and $[W]$. Recall that $\tilde{Q}$ is a closed positive current on $\overline E:= \P(E \oplus \C)$.   By Theorem \ref{th-dieukienHVconictangenetcurrent},  we have 
\begin{align}\label{inemassofQngaself-inter}
\|\tilde{Q}\|  \lesssim \|T\|^{m-l} \vol (W).
\end{align}
Since $Q$ is well-defined on $U$, using Theorem \ref{th-density11agreewithclassical}, we obtain that $\tilde{Q}= \pi^* Q$ on $\pi^{-1}(U)$, where we have identified $\Delta_{m-l+1}$ with $X$.  This combined with (\ref{inemassofQngaself-inter}) yields that 
$$\|Q\| \lesssim \|T\|^{m-l} \vol(W).$$
 Using this and (\ref{ine-dimVlself-interQ}) gives
$$ \sum_{V \in \cali{V}_{T,W}^l}  \big(\nu(T,V)\big)^{m-l} \vol(V) \lesssim \|T\|^{m-l} \vol(W)$$
for every $l$.  Summing the last inequality over $l$ gives the desired inequality.  \\

\noindent
\textbf{Step 2.} We now get rid of the assumption that $T$  has analytic singularities (but we still require $\nu(T,W)=0$).

By Lemma \ref{le-countablelelong} there are at most countably many maximal Lelong level sets of $T$ on $W$. Hence we can write 
$$\cali{V}_{T,W}= \{V_1,\ldots, V_k, \ldots\}.$$
Let $r>0$ be a small constant, and let $A_r$ be the finite subset of $\cali{V}_{T,W}$ consisting of $V$ such that $\nu(T,V) \ge r$. Therefore $A_r$ increases to $\cali{V}_{T,W}$ as $r\to 0$.

Let $(T_\epsilon)_{\epsilon}$ be the sequence of closed positive $(1,1)$-current regularising $T$ given by  Demailly's analytic approximation of psh functions (\cite[Corollary 14.13]{Demailly_analyticmethod}). These currents satisfy the following properties. One has $T_\epsilon \to T$ weakly as $\epsilon \to 0$, and  $\nu(T_\epsilon, \cdot) \le \nu(T,\cdot)$ for every $\epsilon$, and $\nu(T_\epsilon, \cdot) \to \nu(T,\cdot)$ uniformly on $X$ as $\epsilon \to 0$.  It follows that for every $r>0$ there exists an $\epsilon>0$ small enough such that $A_r \subset \cali{V}_{T_\epsilon,W}$. Applying the first part of the proof to $T_\epsilon$ yields that 
$$\sum_{V \in A_r} \big(\nu(T_\epsilon,V)\big)^{m-\dim V}\vol(V) \le c\vol(W) \|T_\epsilon\|(1+ \|T_\epsilon\|)^{m-1},$$
for some constant $c>0$ independent of $T,W$. Letting $\epsilon \to 0$, and then $r \to 0$ gives the desired inequality.
\\

\noindent
\textbf{Step 3.} Finally we treat the general case where $\nu(T,W)>0$. We use \cite[Theorem 1.1]{Demailly_regula_11current} which is a regularisation result of $(1,1)$-currents (see also \cite[Theorem 6.1]{Demailly_appro_chernconnec}). The last theorem allows one to cut down the Lelong level set $E_c:= \{x \in X: \nu(T,x) \ge c\}$ from $T$, where $c>0$ is a constant. Choose $c:= \nu(T,W)$ which is bounded by a uniform constant times $\|T\|$. By  \cite[Theorem 1.1]{Demailly_regula_11current}, we obtain a closed positive $(1,1)$-current $T'$ such that $\|T'\| \le M \|T\|$ ($M$ is a constant independent of $T$), and  $\nu(T',x)= \max\{\nu(T,x)-c, 0\}$, and additionally $T'$ is smooth outside $E_c$ (but we don't need this property at this point). Hence 
$$\cali{V}_{T,W}= \cali{V}_{T',W}, \quad \nu(T',W)=0,$$
 and for every $V \in \cali{V}_{T,W}$ one gets $\nu(T,V)- \nu(T,W)= \nu(T',V)$. Thus the desired inequality follows from Step 2 applied to $T'$ and $W$. 
This finishes the proof.    
\end{proof}


\begin{proof}[Proof of Theorem \ref{theo-apdungRbangT}]
Observe that the length of a Lelong filtration is at most $n-1$, so it is bounded uniformly.  Let $\cali{F}_{T,l}$ be the set of Lelong filtrations of $T$ of length $l$ for $0 \le l \le n-1$. We prove by induction on $l$ that 
\begin{align} \label{ine-quynhaplelongfiltration}
\sum_{\mathbf{V} \in \cali{F}_{T,l}} \vol(\mathbf{V}) \le c \|T\| (1+\|T\|)^{n^{l+1}},
\end{align}
for some constant $c>0$ independent of $T$. If $l=0$, the desired estimate is Corollary \ref{cor-apdungRbangT}. We assume that (\ref{ine-quynhaplelongfiltration}) holds for $l-1$. For every irreducible analytic subset $W$ in $X$, define $\cali{F}_{T,l, W}$ to be the subset of $\cali{F}_{T,l}$ consisting of $\mathbf{V}=(V_j)_{0 \le j \le l}$ such that $V_1= W$. We have
$$\cali{F}_{T,l}= \cup_W \cali{F}_{T,l, W}.$$

 For every $\mathbf{V}= (V_j)_{0 \le j \le l} \in \cali{F}_{T,l,W}$, we associate to it a Lelong filtration $\mathcal{C}_W(\mathbf{V})$ of length $l-1$ by putting
$$\mathcal{C}_W(\mathbf{V}):= (V_j)_{1 \le j \le l}.$$ 
Hence $\mathcal{C}_W$ is a map from $\cali{F}_{T,l,W}$ to $\cali{F}_{T,l-1}$. The image $\cali{F}_{T,l,W}$ under $\mathcal{C}_W$ is the subset of $\cali{F}_{T,l-1}$ consisting of Lelong filtrations starting from $W$. The fiber of $\mathcal{C}_W$ is naturally identified with $\cali{V}_{T,W}$. 
 Applying now Theorem \ref{the-self-inters} to $T$ and $W$ yields
\begin{align*}
\sum_{V \in \cali{V}_{T,W}} (\nu(T,V)-\nu(T,W))^{\dim W- \dim V} \vol(V) &\lesssim  \vol(W) \|T\| (1+\|T\|)^{m-1}.
\end{align*}
Thus
\begin{align*}
\sum_{\mathbf{V}\in \cali{F}_{T,l,W}} \vol(\mathbf{V}) &\lesssim \sum_{\mathbf{V}_1 \in \mathcal{C}_W(\cali{F}_{T,l,W})} \vol(\mathbf{V}_1) \|T\| (1+\|T\|)^{m-1}\\
&\le\sum_{\mathbf{V}_1 \in \mathcal{C}_W(\cali{F}_{T,l,W})} \vol(\mathbf{V}_1)  (1+\|T\|)^{n}
\end{align*}
Summing again over $W$, one get
$$\sum_W \sum_{\mathbf{V} \in \cali{F}_{T,l,W}} \vol(\mathbf{V}) \lesssim \sum_{\mathbf{V}_1 \in \cali{F}_{T,l-1}} \vol(\mathbf{V}_1) (1+\|T\|)^{n}.$$
Since the left-hand side of the last inequality is equal to 
$$\sum_{\mathbf{V} \in \cali{F}_{T,l}} \vol(\mathbf{V}),$$
 the induction hypothesis now implies the desired assertion for $l$. 
This finishes the proof.
\end{proof}

\begin{remark} \label{re-lelongbignef} We explain now how to deduce Theorem \ref{the-self-inters} in the case where $W=X$ from \cite[Theorem 4.1]{Vu_lelong-bignef-quantitative}. Firstly by considering $T/\|T\|$ instead of $T$, we can assume that $T$ of mass $1$.  Observe now that  there is a constant $c>0$ independent of $T$ such that $T+ c \omega$ lies in a fixed compact subset of  K\"ahler cone in $X$. Thus Theorem \ref{the-self-inters} for $W=X$ follows directly from \cite[Theorem 4.1]{Vu_lelong-bignef-quantitative}.     \end{remark}

\begin{proof}[Proof of Corollary \ref{cor-volumesingularset}] We prove first (\ref{ine-vol-sing}). Assume for the moment that $W$ is irreducible. 

Let $T:= [W]$. Note that elements of $\cali{V}_{T,W}$ are irreducible components of $W_1$. Moreover $\nu(T,V) - \nu(T,W) \ge 1$ for $V \in \cali{V}_{T,W}$ because it is a strictly positive integer. This combined with Theorem \ref{the-self-inters} applied to $T:= [W]$ and $W$ yields (\ref{ine-vol-sing}). The case where $W$ is not necessarily irreducible follows by applying the previous case to each irreducible component of $W$. 

It remains to prove (\ref{ine-vol-sing2}).  Let $(W_j)_{1 \le j \le m}$ be the singularity filtration of $W$. Note that $m \le n$. 
Note that $\dim W_{s-1} \le \dim W- (s-1)$. Applying (\ref{ine-vol-sing}) to $W_{s-1}$ in place of $W$ for $1 \le s \le m$, we obtain 
$$\vol(W_{s}) \le c \vol(W_{s-1})(1+\vol(W_{s-1}))^{m- (s-1)},$$
for some constant $c>0$ independent of $W$.
An induction argument thus gives 
$$\vol(W_s) \lesssim \vol(W_l) (1+ \vol(W_l))^{(m-l+2)^{s-l}}$$
for every $0 \le l \le s \le m$. Applying the last inequality to $l=0$ gives 
$$\vol(W_s) \lesssim \vol(W)(1+\vol(W))^{(m+2)^s}$$
for every $1 \le s \le m$. 
This finishes the proof.
\end{proof}

We end the paper with some examples. 

\begin{remark} \label{re-demailly-self} Let $X:= \P^n$ with $n>3$ and $x=[x_0: \cdots:x_n]$ homogeneous coordinates. Let $V_1,V_2,V_3$ be linear subspaces in $\P^n$ such that $\dim V_1= n-1$, $\dim V_2= n-2$, and $\dim V_3=1$, and $V_3 \not \subset V_2 \cup V_1$, and $V_2 \not \subset V_1$.  Let $T_1,T_2,T_3$ be closed positive currents of bi-degree $(1,1)$ such that the set of points with positive Lelong number for $T_j$ is equal exactly to $V_j$, and $\nu(T_j,x)= \nu(T_j,V_j)$ for every $x \in V_j$, and for  $j=1,2,3$ (e.g., $V_2= \{x_0=x_1=0\}$, choose $T_2:= \ddc (|x_0|^2+ |x_1|^2)$).  Choosing positive constants $\lambda_j$ suitably, one see that the current
$$T:= \lambda_1 T_1+ \lambda_2  T_2+ \lambda_3 T_3$$
satisfies 
$$\nu(T,V_1)= 1/100, \quad \nu(T, V_2)= 1/100, \quad \nu(T, V_3)=1.$$
Hence $V_1,V_2,V_3$ are all maximal Lelong upper level sets for $T$. 
On the other hand, if $E_c:=\{x \in \P^n, \nu(T,x) \ge c\}$, 
then  $E_c=V_1 \cup V_2\cup V_3$ if $c \in (0, 1/100]$, and $E_c=V_3$ if $c \in (1/100,1]$, and $E_c= \varnothing$ if $c>1$. One sees that the maximal set $V_2$ is not taken into account in the self-intersection inequality in \cite[Theorem 1.7]{Demailly_regula_11current} as well as its generalizations in \cite{Meo-auto-inter,Vigny-LelongSkoda} (because $V_2$ is of codimension 2, but only appears in $E_{1/100}$ which is of codimension 1).
\end{remark}

\begin{example}\label{ex-phaidungrelative} Let $X= \P^3$, $W\approx \P^2$ a hyperplane in $\P^3$ and $D$ a hypersurface of degree $d$ in $\P^3$ so that $D$ intersects $W$ at a smooth curve $V$ (of degree $d$). Let $T:= \epsilon [D]+  [W]$ for some small constant $\epsilon>0$. We see that $\cali{V}_{T,W}= \{V\}$ and $\vol(V)= d$ and 
$$\nu(T,V) \vol(V)= (1+\epsilon) d,\quad  c\vol(W) \|T\|(1+ \|T\|)\le c(2+\epsilon d)^2< (1+ \epsilon)d$$
if $d$ is big enough and $\epsilon:= 1/d$ (recall $c$ is independent of $T,W$). Thus  the inequality (\ref{ine-dimVlself-interdesire}) fails to be true in general if $\nu(T,V)- \nu(T,W)$ is replaced by $\nu(T,V)$. 
\end{example}

\bibliography{biblio_family_MA,biblio_Viet_papers}
\bibliographystyle{siam}

\bigskip

\noindent
\Addresses
\end{document}